   \documentclass[11pt,letterpaper]{amsart}
\usepackage{amsmath, amssymb, amsfonts, graphicx, pgfplots, tikz, tikz-cd, stmaryrd}

\newtheorem{Theorem}{Theorem}[section]
\newtheorem{prop}[Theorem]{Proposition}

\def\beq#1#2\eeq{%
        \begin{equation}%
        \label{#1}%
            #2%
        \end{equation}%
   }

\usepackage{color}
\usepackage{graphicx}
\usepackage{epstopdf}

\title[Theta divisors and permutohedra]{Theta divisors and permutohedra}

\author{V.M. Buchstaber}\address{Steklov Mathematical Institute and Moscow State University, Russia}
\email{buchstab@mi-ras.ru}

\author{A.P. Veselov}
\address{Department of Mathematical Sciences,
Loughborough University, Loughborough LE11 3TU, UK}
\email{A.P.Veselov@lboro.ac.uk}

\begin{document}

\maketitle

\begin{abstract}
 We establish an intriguing relation of the smooth theta divisor $\Theta^n$ with permutohedron $\Pi^n$ and the corresponding toric variety $X_\Pi^n.$ In particular, we show  that the generalised Todd genus of the theta divisor $\Theta^n$ coincides with $h$-polynomial of permutohedron $\Pi^n$ and thus is different from the same genus of $X_\Pi^n$ only by the sign $(-1)^n.$ As an application we find all the Hodge numbers of the theta divisors in terms of the Eulerian numbers. We reveal also interesting numerical relations between theta-divisors and Tomei manifolds from the theory of the integrable Toda lattice.
\end{abstract}



\section{Introduction}

The theta divisors are very classical object of study going back to Riemann (see  \cite{GruH}). They can be given as the zero set of the Riemann $\theta$-function
of a principally polarised abelian varieties $A^{n+1}.$ It is known after Andreotti and Mayer \cite{AM67} that the corresponding theta divisor $\Theta^n \subset A^{n+1}$ is a smooth projective variety for a general ppav $A^{n+1}$. It has natural subvarieties given by the smooth intersections
\beq{inters0}
\Theta_{k}^{n-k}=\Theta^n\cap \Theta^n(a_1) \cap \dots \Theta^n(a_k)
\eeq
of $\Theta^n$ with $k$ general translates $\Theta^n(a_i), \, a_i \in A^{n+1}$ of the theta divisor $\Theta^n.$

Recently it was discovered that the theta divisors $\Theta^n$ play an important role in the theory of complex cobordisms \cite{BV2020}. Namely, we proved that $\Theta^n$ can be chosen as irreducible algebraic representatives of the coefficients of the Chern-Dold character in complex cobordisms and described the action of the Landweber-Novikov operations on them in terms of $\Theta_{k}^{n-k}.$

 The aim of this paper is to establish a link of the theta divisor $\Theta^n$ with combinatorics of permutohedron $\Pi^n$ and the corresponding permutohedral toric variety $X_\Pi^n$, which we found very intriguing. 
Recall that permutohedron $\Pi^n$ is a simple $n$-dimensional lattice polytope, which we can choose to be the convex hull of the points $\sigma(\rho)\in \mathbb R^{n+1}, \,\rho=(1,2,\dots,n, n+1), \, \sigma \in S_{n+1}.$ 

Our first result computes the Todd genus of $\Theta_{k}^{n-k}$ in terms of combinatorics of the permutohedron.

\begin{Theorem} The Todd genus of the self-intersection of theta divisors 
\beq{rel}
Td(\Theta_{k}^{n-k})=(-1)^{n-k} f_{n-k}(\Pi^n)
\eeq
up to a sign coincides with the number  $f_{n-k}(\Pi^n)$ of the codimension $k$ faces of permutohedron $\Pi^n$.
\end{Theorem}

Since it is known that 
$
f_{n-k}(\Pi^n)=(k+1)! \, S(n+1,k+1),
$
where $S(n,k)$ are the Stirling numbers of second kind \cite{Stanley},
we have the formula
$$
Td(\Theta_{k}^{n-k})=(-1)^{n-k} (k+1)! \, S(n+1,k+1).
$$


Our second result reveals the relation of the two-parameter Todd genus $Td_{s,t}$ of theta divisor $\Theta^n$ with the $h$-polynomial of permutohedron $\Pi^n$. 

Recall that the $h$-polynomial $h_{P^n}(s,t)$ of $n$-dimensional simple polytope $P^n$ is related to $f$-polynomial $f_{P^n}(s,t)=\sum_{k=0}^n f_{n-k}(P^n)s^{n-k}t^k$ by simple change
$$
h_{P^n}(s,t):=f_{P^n}(s-t,t).
$$

The two-parameter Todd genus $Td_{s,t}$ is a homogeneous version of the Hirzebruch $\chi_y$-genus introduced by Krichever \cite{Krich}. It corresponds to the generating series
$$
Q(x)=\frac{x(se^{tx}-te^{sx})}{e^{sx}-e^{tx}}.
$$
When $s=y, t=-1$ it reduces to the $\chi_y$-genus \cite{Hirz}.

\begin{Theorem} The two-parameter Todd genus $Td_{s,t}(\Theta^n)$ of the theta divisor $\Theta^n$ coincides with the $h$-polynomial of permutohedron $\Pi^n$:
\beq{equalint}
Td_{s,t}(\Theta^n)=h_{\Pi^n}(s,t).
\eeq
In particular, the $\chi_y$-genus of theta divisors is
\beq{equaly}
\chi_y(\Theta^n)=h_{\Pi^n}(y,-1)=(-1)^nA_{n+1}(-y),
\eeq
where $A_n(y)$ are the classical Eulerian polynomials \cite{Stanley}.
\end{Theorem}


As an application we compute all the Hodge numbers $h^{p,q}(\Theta^n)$.

\begin{Theorem} 
The Hodge numbers $h^{p,q}$ of theta divisor $\Theta^n$ with $p+q \neq n$ are given explicitly by
$$
h^{p,q}(\Theta^{n})=h^{n-p,n-q}(\Theta^n)={n+1 \choose  p} {n+1 \choose q}, \quad p+q\leq n-1.
$$
When $p+q=n$ we have
$
h^{p,n-p}(\Theta^{n})=A_{n+1,p}-S_{n,p},
$
where $A_{n,p}$ are the Eulerian numbers and 
$$
S_{n,p}=(-1)^{p} {n+2 \choose  p+1} \left[(-1)^p\frac{2p-n}{n+2}{n+1 \choose  p}+\sum_{k=0}^{p-1}(-1)^k{n+1 \choose k}\right].
$$
In particular, 
$$
h^{0,n}(\Theta^{n})=n+1, \quad h^{1,n-1}(\Theta^n)=2^{n+1}-(n+2)+\frac{n^2(n+1)}{2}.
$$
\end{Theorem}

The explicit forms of the Hodge diamonds of $\Theta^n$ for $n=2,3,4$ are shown in Section 5 below.

As a corollary of our results we establish an interesting duality between theta divisor $\Theta^n$ and the permutohedral variety $X_\Pi^n$, which is the toric variety determined by $\Pi^n$ \cite{Ful}.

\begin{Theorem} The Betti number $b_{2k}(X_\Pi^n)$ of the permutohedral variety coincides up to a sign with the Hirzebruch $\chi^k$-genus of the theta divisor $\Theta^n$:
$$
b_{2k}(X_\Pi^n)=(-1)^{n-k}\chi^k(\Theta^n).
$$
The same is true for the two-parameter Todd genus of these two varieties:
$$
Td_{s,t}(X_\Pi^n)=(-1)^n Td_{s,t}(\Theta^n).
$$
\end{Theorem}

This might suggest that the corresponding cobordism classes are related by $[X_\Pi^n]=(-1)^n[\Theta^n].$ This indeed works for $n\leq 2$, but already for $n=3$ this is not the case. In fact we provide a formula expressing the cobordism class $[X_\Pi^n]$ in terms of the theta divisors (see Theorem 6.1 below).

 In the rest of the paper we discuss the connection of $\Theta^n$ and $X_\Pi^n$ with two other manifolds appeared in relation with integrable Toda lattice and known to be related to permutohedra.

The first one is the Tomei manifold $M_T^n$, which is a real $n$-dimensional manifold consisting of the real symmetric tridiagonal matrices with given spectrum. Tomei \cite{Tomei} used the Toda flows to show that $M_T^n$  can be glued from $2^n$ copies of permutohedron and computed its Euler characteristic. 
We use this to show that, in particular, the Euler characteristic of the Tomei manifold equals the signatures of both $X_\Pi^n$ and $\Theta^n$: $$\chi(M_T^n)=\tau(X_\Pi^n)=\tau(\Theta^n).$$ 

We show also that the Hermitian version of Tomei manifold $M^{4n}_{HT}$, studied by Bloch, Flaschka and Ratiu \cite{BFR}, is not diffeomorphic to any symplectic manifold $M^{4n}$ with Hamiltonian action of torus $T^{2n}$ and that $M^{4}_{HT}$ does not admit any almost complex (and hence, any symplectic) structure.

\section{Theta divisors and complex cobordisms }

In this section we describe the results about theta divisors and their role in complex cobordism theory mainly following \cite{BV2020}.
 
Let $A^{n+1}=\mathbb C^{n+1}/\Gamma$ be a principally polarised abelian variety (ppav) with lattice $\Gamma$ generated by the columns of the $(n+1)\times 2(n+1)$ matrix $(I, \,\, \tau)$ with complex symmetric $(n+1)\times (n+1)$ matrix $\tau$ having positive imaginary part \cite{GH}. Its polarisation line bundle $L$ has one-dimensional space of sections generated by the classical Riemann $\theta$-function
\beq{theta}
\theta(z, \tau)=\sum_{l\in \mathbb Z^{n+1}}\exp [\pi i (l, \tau l) + 2\pi i (l, z)], \, z \in \mathbb C^{n+1}.
\eeq
Andreotti and Mayer \cite{AM67}) proved that the corresponding theta divisor $\Theta^n \subset A^{n+1}$ given by $\theta(z,\tau)=0$ is smooth for a general ppav $A^{n+1}$.

In particular, for $n=1$ a generic abelian surface $A^2$ is the Jacobi variety of a smooth genus 2 curve $\mathcal C$ with theta divisor $\Theta^1 \cong \mathcal C$. 
For $n=2$ an indecomposable $A^3$ is Jacobi variety of a genus 3 curve $\mathcal C$ and $\Theta^2 \cong S^2(\mathcal C)$ is smooth for all non-hyperelliptic curves $\mathcal C$. For $n\geq 3$ the general case of $A^{n+1}$  is not Jacobian, and the theta divisor is smooth outside a locus in the moduli space of the abelian varieties of complex codimension 1 (see more on this in \cite{BL,GruH}).

 The topology of smooth theta divisor does not dependent on the choice of such abelian variety and can be studied using the Lefschetz hyperplane theorem (see \cite{IW, BV2020}). In particular, the Euler characteristic is 
 \beq{chi}
 \chi(\Theta^n)=(-1)^n(n+1)!,
 \eeq 
 the fundamental group $\pi_1(\Theta^n)=\pi_1(A^{n+1})=\mathbb Z^{2n}$ for $n\geq 2,$
the  Betti numbers of $\Theta^n$ are
 \beq{betti1}
b_k(\Theta^n)=b_k(A^{n+1})={2n+2 \choose k}=b_{2n-k}(\Theta^n), \,\, k<n,
\eeq
 \beq{betti2}
b_{n}(\Theta^n)=(n+1)!+\frac{n}{n+2} {2n+2 \choose n+1}=(n+1)!+n C_{n+1},
\eeq
where $C_n=\frac{1}{n+1} {2n \choose n}$ is the $n$-th Catalan number \cite{Stanley}.

The theta divisors have natural subvarieties given by the intersections 
\beq{inters2}
\Theta_{k}^{n-k}=\Theta^n\cap \Theta^n(a_1) \cap \dots \Theta^n(a_k)
\eeq
of $\Theta^n$ with $k$ general translates $\Theta^n(a_i), \, a_i \in A^{n+1}$ of the theta divisor $\Theta^n.$ 
 For all $k<n$ and general $a_i \in A^{n+1}, \, i=1,\dots,k$ the variety $\Theta_{k}^{n-k}$ is smooth and irreducible of general type \cite{BV2020}.

%
%
%
%
%
%

In \cite{BV2020} it was discovered that the theta divisors are playing a very special role in the complex cobordism theory \cite{Stong}.

Let $M^m$ be a smooth closed real oriented manifold. By {\it stable complex structure} (or, simply $U$-{\it structure}) on $M^m$ we mean an isomorphism of the real oriented vector bundles
$TM^m\oplus (2N-m)_{\mathbb R}\cong r\xi,$
where $TM^m$ is the tangent bundle of $M^m$, $(2N-m)_{\mathbb R}$ is trivial naturally oriented real $(2N-m)$-dimensional bundle over $M^m$, $\xi$ is a complex vector bundle over $M^m$ and $r\xi$ is its real form. A manifold $M^m$ with a chosen $U$-structure is called $U$-manifold.
 Note that a complex structure on $\xi$ 
 determines complex structure in the stable normal bundle $\nu M^m.$
 
 Two closed smooth real oriented $m$-dimensional $U$-manifolds $M_1$ and $M_2$ are called {\it $U$-cobordant} if there exists a real $(m+1)$-dimensional $U$-manifold $W$ with boundary such that the boundary $\partial W$ is a disjoint union of $M_1^m$ with given orientation and $M_2^m$ with the opposite orientation, and such that the restriction of the stable complex normal bundle $\nu W$ to $M_i$ coincides with the stable complex normal bundles $\nu M_i, \, i=1,2.$ 
 
 The disjoint union and direct product of $U$-manifolds define the commutative graded cobordism ring  $\Omega_U=\sum_{m\geq 0}\Omega^{-m}_U$, where $\Omega^{-m}_U$ is the group of cobordism classes of $m$-dimensional $U$-manifolds. 
 
 The cobordism ring $\Omega_U$ was computed by Milnor \cite{Mil} and Novikov \cite{Nov60}, who proved that  $\Omega_U=\mathbb Z[y_1,\dots,y_n,\dots],\, \deg y_n=-2n$ is the graded polynomial ring of infinitely many generators $y_n, \, n \in \mathbb N.$ 
  The bordism ring $\Omega^U$ is dual to $\Omega_U$. 
  
  There exist corresponding homology $U_*(X)$ (bordisms) and cohomology $U^*(X)$ (cobordisms) theories with $U_*(pt)=\Omega^U$ and $U^*(pt)=\Omega_U$ respectively \cite{Nov}. Geometric construction of cobordisms, using the ideas from both algebraic topology and algebraic geometry was given by Quillen in \cite{Quillen}.

By definition, the Chern-Dold character $ch_U$  is a natural multiplicative transformation of cohomology theories
$$
ch_U: U^*(X)\to H^*(X, \Omega_U\otimes \mathbb Q),
$$
where $U^*(X)$ is the complex cobordism ring of a $CW$-complex $X.$ 

Let $u \in U^2(\mathbb CP^\infty)$ and $z\in H^2(\mathbb CP^\infty)$ be the first Chern classes of the universal line bundle
on $\mathbb CP^\infty$ in the complex cobordisms and cohomology theory respectively.
The Chern-Dold character is uniquely defined by its action on $u$:
$$
ch_U: u \to \beta(z), \quad \beta(z):=z+\sum_{n=1}^\infty[\mathcal B^{2n}]\frac{z^{n+1}}{(n+1)!},
$$
where $\mathcal B^{2n}$ are certain $U$-manifolds, characterised by their properties in \cite{B-1970}. 
In \cite{BV2020} we proved that as the representatives of these cobordism classes one can use the theta divisors:
\beq{CDnew}
\beta(z)=z+\sum_{n=1}^\infty[\Theta^n]\frac{z^{n+1}}{(n+1)!}.
\eeq


As a corollary we have the following explicit expression of the exponential generating function of any Hirzebruch genus $\Phi$ of theta divisors:
\beq{hirzeb}
\Phi(\Theta, z):=\sum_{n=0}^\infty\Phi(\Theta^n)\frac{z^{n+1}}{(n+1)!} = \frac{z}{Q(z)},
\eeq
where $Q(z)=1+\sum_{n\in \mathbb N}a_nz^n$ is the characteristic power series of Hirzebruch genus $\Phi$ (see \cite{Hirz,BV2020}).





Let us introduce the generating function of the Todd genera of the self-intersections of theta divisors as
\beq{F}
Td_\Theta(x,b,t):=\sum_{k,n\geq 0, k\leq n} Td(\Theta_{k}^{n-k})\frac{b^{n-k}t^k x^{n+1}}{(n+1)!}.
\eeq

We can show now that it can be viewed also as the generating function of the $K$-theory Chern numbers \cite{CF} of theta divisors.
Indeed, Conner and Floyd \cite{CF} constructed the transformation $\mu_c: U^*(X) \to K^*(X)$ 
of complex cobordisms to complex $K$-theory, related to Riemann-Roch theorem in algebraic geometry \cite{Hirz}. When $X=pt,$ we have $\mu_c: \Omega^*_U \to K^*(pt)=\mathbb Z[b,b^{-1}]$, where $b$ is the Bott periodicity operator with $\deg b=-2,$ defined by  
\beq{RR}
\mu_c([M^{2n}])=Td(M^{2n})b^n.
\eeq
Using the complex cobordism theory one can define the {\it K-theory Chern numbers} $c^K_\lambda(M^{2n}) \in \mathbb Z[b,b^{-1}]$ of any $U$-manifold $M^{2n}$ as follows
\beq{KCH}
c_\lambda^K(M^{2n}):=Td(S_\lambda[M^{2n}])b^{n-|\lambda|},
\eeq
where $\lambda$ is a partition with $|\lambda| \leq n$ and $S_\lambda$ is the Landweber-Novikov operation \cite{Nov}.
If $\lambda=\emptyset$, then $S_\lambda=Id$  and we have formula (\ref{RR}) for $\mu_c=c^K_{\emptyset}$.

 In \cite{BV2020} we have described explicitly the action of the Landweber-Novikov operations on the theta divisors.

\begin{Theorem} (\cite{BV2020})
If $\lambda$ is not a one-part partition, then $S_\lambda[\Theta^n]=0,$ while for $\lambda=(k), \, k\leq n$ we have
\beq{LN}
S_{(k)}[\Theta^n]=[\Theta_{k}^{n-k}],
\eeq
where $\Theta_{k}^{n-k}$ is the intersection of shifted theta divisors (\ref{inters0}).
\end{Theorem}

In combination with (\ref{KCH}) this implies the following result.

\begin{prop} 
The generating function of the K-theory Chern numbers of the theta divisors
\beq{KTheta}
K_{\Theta}(x,t):=\sum_{k,n\geq 0, k\leq n} c_{(k)}^K(\Theta^n)\frac{t^k x^{n+1}}{(n+1)!}
\eeq
coincides with the generating function $Td_\Theta(x,b,t).$
\end{prop}

Now we give an explicit formula for both these generating functions.

\begin{prop} 
The generating functions $Td_{\Theta}(x,b,t)$ and $K_{\Theta}(x,t)$ can be given explicitly as
\beq{FE}
Td_\Theta(x,b,t)=K_{\Theta}(x,t)=\frac{1-e^{-bx}}{b-t(1-e^{-bx})}.
\eeq
\end{prop}

\begin{proof}
We use the fact that Chern-Dold character $ch_U$ commutes with Landweber-Novikov operations:
\beq{comm}
S_{(k)}\circ ch_U=ch_U \circ S_{(k)} 
\eeq
(see \cite{B-1970}) and that $S_{(k)}u=u^{k+1},$
where $u \in U^2(\mathbb CP^\infty)$ as before is the first Chern class of the universal line bundle
on $\mathbb CP^\infty$ in the complex cobordisms. Applying this to $u \in U^2(\mathbb CP^\infty)$  and using the relations (\ref{CDnew}) and (\ref{LN}) we have 
$$
S_{(k)}\circ ch_U (u)=S_{(k)} (\beta(z))=\sum_{n\geq 0} S_{(k)}([\Theta^n])\frac{z^{n+1}}{(n+1)!}=\sum_{n\geq 0} [\Theta_k^{n-k}]\frac{z^{n+1}}{(n+1)!}.
$$
On the other hand since
$
ch_U \circ S_{(k)}(u)= ch_U(u^{k+1})=\beta(z)^{k+1},
$ 
we have
\beq{123}
\sum_{n\geq 0} [\Theta_k^{n-k}]\frac{z^{n+1}}{(n+1)!}=(\sum_{n\geq 0} [\Theta^n]\frac{z^{n+1}}{(n+1)!})^{k+1}.
\eeq
Applying now the Riemann-Roch transformation (\ref{RR}) to both sides of (\ref{123}) and using the fact that $Td(\Theta^n)=(-1)^n$ (see \cite{BV2020}), we have 
$$
\sum_{n\geq 0} Td(\Theta_{k}^{n-k})b^{n-k}\frac{z^{n+1}}{(n+1)!}=\left(\sum_{n\geq 0} (-1)^n b^n \frac{z^{n+1}}{(n+1)!}\right)^{k+1}=\left(\frac{1-e^{-bz}}{b}\right)^{k+1}.
$$
Multiplying both sides by $t^k$ and adding over $k\leq n$ we have the relation (\ref{FE}) and the claim.
\end{proof}

Remarkably the same generating function describes the combinatorics of the permutohedron.

\section{Topology of theta divisors and combinatorics of permutohedra}

Recall that {\it permutohedron} (aka permutahedron) $\Pi^n$ is simple convex polytope, which is a convex hull of the points $\sigma(x), \sigma \in S_{n+1}$, being the orbit of the symmetric group $S_{n+1}$, acting on a generic point $x \in \mathbb R^{n+1},$ which can be chosen to be $\rho=(1,2,\dots,n,n+1).$ 

It can also be described as the Newton polytope of the Vandermonde polynomial
$\prod_{1\leq i<j\leq n}(x_i-x_j).$ 
For $n=2$ we have hexagon, for $n=3$ - the truncated octahedron shown on Fig. 1.

\begin{figure}[h]
\includegraphics[width=20mm]{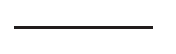} \includegraphics[width=30mm]{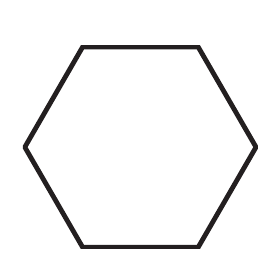} \includegraphics[width=30mm]{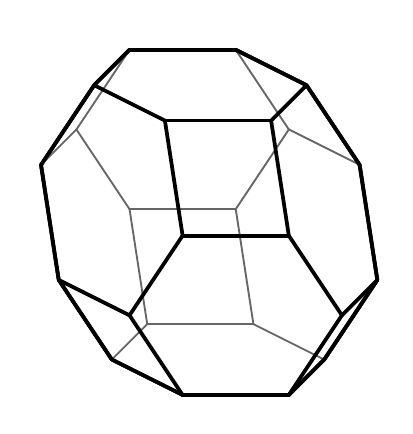}
\caption{Permutohedra in dimension 1,2 and 3.}
\end{figure} 

Its combinatorics is well-studied, see e.g. \cite{GKZ, Post, Ziegler} and references therein.
In particular, it is known that the number $f_{n-k}(\Pi^n)$ of faces of dimension $n-k$ (or, codimension $k$) can be given as
\beq{st}
f_{n-k}(\Pi^n)=(k+1)! \, S(n+1,k+1),
\eeq
where $S(n,k)$ are the {\it Stirling numbers of second kind} \cite{Stanley}. 
These numbers can be computed recursively:
$$
S(n+1,k)=kS(n,k)+S(n, k-1),
$$
with $S(0,0)=1$ and  $S(n,0)=S(0,n)=0$ for $n>0.$

Consider the corresponding {\it $f$-polynomial} of permutohedron $\Pi^n$
\beq{fpol}
f_{\Pi^n}(s,t):=\sum_{k=0}^n f_{n-k}(\Pi^n)s^{n-k}t^k,
\eeq 
where $f_{n-k}(\Pi^n)$ is the number of faces of $\Pi^n$ of dimension $n-k$:
$$
f_{\Pi^1}(s,t)=s+2t, \,\,\, f_{\Pi^2}(s,t)=s^2+6ts+6t^2, \,\,\, f_{\Pi^3}(s,t)=s^3+14s^2t+36s t^2+24t^3, \dots.
$$
Let 
$$
F_\Pi(x,s,t):=\sum_{n\geq 0} f_{\Pi^n}(s,t) \frac{x^{n+1}}{(n+1)!}=\sum_{k,n\geq 0, \, k\leq n} f_{n-k}(\Pi^n)s^{n-k}t^k\frac{x^{n+1}}{(n+1)!}
$$
be their generating function, which can also be considered as the generating function of the face numbers of all permutohedra.

\begin{prop} 
The Todd generating function $Td_{\Theta}(x,b,t)$ of the intersections of theta divisors (\ref{F}) coincides with the permutohedral face generating function $F_\Pi(x,s,t)$ after the substitution $s=-b$:
\beq{Rel}
Td_\Theta(x,b,t)=F_\Pi(x,-b,t).
\eeq
\end{prop}

\begin{proof}
We use the results of \cite{B-2008}, where it was shown that the generating function of the face numbers of permutohedra can be given explicitly as
\beq{Fex}
F_\Pi(x,s,t)=\frac{e^{sx}-1}{s-t(e^{sx}-1)}.
\eeq
This follows from the recursive formula for the boundary $d\Pi^n$ of the permutohedron $\Pi^n$
$$
d\Pi^n=\sum_{i+j=n-1}{n+1 \choose i+1}\Pi^i \Pi^j
$$
(see formula (18), Theorem 17 and Corollary 21 in \cite{B-2008}).

Using this we have the relation
$$
F_\Pi(x,-b,t)=\frac{e^{-bx}-1}{-b-t(e^{-bx}-1)}=\frac{1-e^{-bx}}{b-t(1-e^{-bx})}=Td_\Theta(x,b,t),
$$
which implies the claim.
\end{proof}

As a corollary we have the proof of Theorem 1.1, claiming that the Todd genus of $\Theta_k^{n-k}$ up to a sign coincides with the number of faces of permutohedron $\Pi^n$ of codimension $k:$
\beq{rel}
Td(\Theta_{k}^{n-k})=(-1)^{n-k} f_{n-k}(\Pi^n).
\eeq

%
%
%

In particular, using the explicit form of the Stirling numbers \cite{Stanley}
$$
S(n+1,n)={n+1 \choose 2}, \quad S(n+1,2)=2^n-1,
$$
we have
$$
Td(\Theta_{n-1}^{1})=-n\frac {(n+1)!}{2}, \,\, Td(\Theta_{1}^{n-1})=(-1)^{n-1}(2^{n+1}-2),
$$
so $\Theta_{n-1}^{1}$ is a curve of genus 
$$
g=1+n\frac {(n+1)!}{2}
$$
in agreement with \cite{BV2020}.

\section{The two-parameter Todd genus of theta divisors and $h$-polynomials of permutohedra}

Consider the formal group depending on two parameters $a$ and $b$:
\beq{FG}
x, y \to F(x,y)=\frac{x+y+axy}{1-bxy}.
\eeq
Its exponential can be given as
\beq{expo}
\beta(x)=\frac{e^{sx}-e^{tx}}{se^{tx}-te^{sx}},
\eeq
where parameters $s$ and $t$ are related to $a$ and $b$ as
$$
a=s+t, \, b=st.
$$
When $a=-1$, $b=0$ (corresponding to $s=-1$, $t=0$) we have the formal group with the operation
$$
x, y \to F(x,y)=x+y-xy
$$
with the exponential
$$
\beta(x)=1-e^{-x},
$$
corresponding to the classical Todd genus \cite{Hirz}.

Let $Td_{s,t}$ be the corresponding two-parameter Todd genus, corresponding to the formal group (\ref{FG}) and consider the exponential generating function of this genus for the theta divisors:
\beq{H}
Td^\Theta_{s,t}(x):=\sum_{n\geq 0} Td_{s,t}(\Theta^n) \frac{x^{n+1}}{(n+1)!}.
\eeq

In \cite{BV2020} we have proved that the exponential generating function of any Hirzebruch genus $\Phi$ of theta divisors:
\beq{hirzeb}
\Phi(\Theta, z):=\sum_{n=0}^\infty\Phi(\Theta^n)\frac{z^{n+1}}{(n+1)!} = \frac{z}{Q(z)}=\beta(z)
\eeq
where $Q(x)$ is the generating power series of genus $\Phi$ and $\beta(x)$ is the exponential $\beta$ of the corresponding formal group. In particular, in our case we have
\beq{Hex}
Td^\Theta_{s,t}(x)=\frac{e^{sx}-e^{tx}}{se^{tx}-te^{sx}}.
\eeq

Remarkably the same generating function describes the $h$-polynomials of permutohedra. 

Recall that {\it $h$-polynomial} $h_{P^n}(s,t)$ of $n$-dimensional simple polytope $P^n$ is related to $f$-polynomial $f_{P^n}(s,t)$ by simple change
\beq{hf}
h_{P^n}(s,t):=f_{P^n}(s-t,t)=\sum_{k=0}^n h_{n-k}(P^n)s^{n-k}t^k.
\eeq
The $h$-polynomials are known to be symmetric (Dehn-Sommerville relations): $$h_{P^n}(s,t)=h_{P^n}(t,s),$$
 and their coefficients $h_{k}(P^n)=h_{n-k}(P^n)=\dim H^{2k}(X_P^n)$  are even Betti numbers of the corresponding toric varieties, see \cite{Ful}.

 Now we are ready to prove Theorem 1.2, namely that the two-parameter Todd genus $Td_{s,t}(\Theta^n)$ of the theta divisor $\Theta^n$ coincides with the $h$-polynomial of permutohedron $\Pi^n$:
\beq{equal}
Td_{s,t}(\Theta^n)=h_{\Pi^n}(s,t)
\eeq
and, as a corollary, that
\beq{equal2}
\chi_y(\Theta^n)=(-1)^n A_{n+1}(-y),
\eeq
where $A_{n+1}(y)$ is the classical Eulerian polynomial.


To prove the first part we use the results
of \cite{B-2008}, where it was shown that the generating function of the $h$-polynomials of the permutahedra 
\beq{H}
H_\Pi(x,s,t):=\sum_{n\geq 0} h_{\Pi^n}(s,t) \frac{x^{n+1}}{(n+1)!}=\frac{e^{sx}-e^{tx}}{se^{tx}-te^{sx}}
\eeq
and thus
$H_\Pi(x,s,t)=Td^\Theta_{s,t}(x)$, implying (\ref{equal}).

To prove the second claim recall that the {\it Eulerian number} $A_{n,k}$ is the number of permutations from $S_{n}$ with $k$ descents, see e.g. \cite{Stanley}.
These numbers have the symmetry $A_{n,k}=A_{n, n-k-1}$ and satisfy the recurrence 
$$
A_{n,k}=(n-k)A_{n-1,k-1}+(k+1)A_{n-1,k}.
$$
They can be given also as the sum
\beq{expli}
A_{n,m}=\sum_{k=0}^m(-1)^k {n+1 \choose k} (m+1-k)^n.
\eeq
The corresponding polynomials 
$A_n(s)=\sum_{k=0}^{n-1}A_{n,k} s^k$
were introduced by Euler in 1755 by the relation
$$
\sum_{k=1}^\infty k^nt^n=\frac{tA_n(t)}{(1-t)^{n+1}}.
$$
They can be computed recursively by
$$
A_{n+1}(t)=[t(1-t)\frac{d}{dt} +nt+1]A_n(t), \quad A_1=1:
$$
$$A_1=1, \,\, A_2=s+1,\,\, A_3=s^2+4s+1,\,\,
A_4=s^3+11s^2+11s+1,$$
$$A_5=s^4+26s^3+66s^2+26s+1, \,\, A_6=s^5+57s^4+302s^3+302s^2+57s+1.
$$
The generating function of Eulerian polynomials is known after Euler to be
\beq{genep}
\sum_{n\geq 0}A_n(s)\frac{x^n}{n!}=\frac{s-1}{s-e^{(s-1)x}}.
\eeq
Consider
$$
A(x,s):=\sum_{n\geq 0}A_{n+1}(s)\frac{x^{n+1}}{(n+1)!}=\frac{s-1}{s-e^{(s-1)x}}-1=\frac{e^{sx}-e^x}{se^x-e^{sx}}.
$$
Replacing here $x$ by $tx$ and $s$ by $s/t$ we have the equality (see \cite{B-2008})
$$
\sum_{k, n\geq 0, k\leq n} A_{n+1,k}s^kt^{n-k}\frac{x^{n+1}}{(n+1)!}=\frac{e^{sx}-e^{tx}}{se^{tx}-te^{sx}}=Td^\Theta_{s,t}(x).
$$
 Setting now $s=y, t=-1$ we have formula (\ref{equal2}), completing the proof of Theorem 1.2.

\section{Application: Hodge numbers of the theta-divisors}

Let $H^{p,q}(X)$ be the Dolbeault cohomology group of a complex $n$-dimensional manifold $X$ and $h^{p,q}(X)=\dim H^{p,q}(X).$

Following Hirzebruch \cite{Hirz} consider the index of the elliptic operator $$\bar \partial: \Omega^{p,q}(X)\to \Omega^{p, q+1}(X)$$ for fixed $p$ and consider the corresponding index
\beq{index}
\chi^p(X):=\sum_{q=0}^n(-1)^q h^{p,q}(X).
\eeq
When $p=0$ we have the {\it holomorphic Euler characteristic}, 
 which is known to coincide with the Todd genus of $X$:
$
\chi^0(X)=Td(X)
$
and is related to the {\it arithmetic genus} $\chi_a(X)$ by the formula
$$
\chi_a(X)=(-1)^n(\chi^0(X)-1)
$$
(see  \cite{Hirz}).
To compute other $\chi^p(X)$ introduce the generating polynomial
$$
\chi_y(X):=\sum_{p=0}^n \chi^p(X)y^p.
$$

\begin{Theorem} (Hirzebruch \cite{Hirz})
The value of $\chi_y(X)$ can be given by the Hirzebruch genus with the generating power series
\beq{geny}
Q(x)=\frac{x(1+ye^{-x(1+y)})}{1-e^{-x(1+y)}}.
\eeq
\end{Theorem}

Applying now our general formula (\ref{hirzeb}) we have
$$
\sum_{n=1}^\infty\chi_y(\Theta^n)\frac{z^{n+1}}{(n+1)!}=\frac{1-e^{-x(1+y)}}{1+ye^{-x(1+y)}}.
$$
Since
$$
\frac{1-e^{-x(1+y)}}{1+ye^{-x(1+y)}}=\frac{e^{yx}-e^{-x}}{e^{yx}+ye^{-x}}
$$
we see that we have a particular case of two-parameter Todd genus $Td_{s,t}$ with $s=y, t=-1.$ Thus we have the following result.

\begin{prop} 
The $\chi_y$-genus of the theta divisor $\Theta^n$  can be given as
\beq{genye}
\chi_y(\Theta^n)=(-1)^n A_{n+1}(-y),
\eeq
where $A_n(s)$ is the Eulerian polynomial. In particular,
\beq{chip}
\chi^p(\Theta^n)=(-1)^{n-p}A_{n+1,p},
\eeq
where $A_{n,p}$ are the Eulerian numbers.
\end{prop}

When $y=0$ we have the classical Todd genus 
$$
\chi^0(\Theta^n)=Td(\Theta^n)=A_{n+1,n}(-1)^n=(-1)^n
$$
in agreement with \cite{BV2020}.
When $y=-1$ we have the Euler characteristic $$\chi(\Theta^n)=(-1)^n A_{n+1}(1)=(-1)^n(n+1)!$$
again in agreement with  \cite{BV2020}.
Finally when $y=1$ we have the formula for the signature of the theta divisor for even $n$ 
\beq{tau}
\tau(\Theta^n)=\sum_{k=0}^n (-1)^k A(n+1,k)= \frac{2^{n+2}(2^{n+2}-1)}{n+2} B_{n+2},
\eeq
where $B_{n}$ are the classical {\it Bernoulli numbers}, again in agreement with \cite{BV2020}.

We can use this to compute the {\it Hodge numbers} $h^{p,q}(\Theta^n)=\dim H^{p,q}(\Theta^n),$ where $$H^{p,q}(M)=H^{p,q}_{\bar \partial}(M)=H^q(M, \Omega^p_{M})$$  are the Dolbeault cohomology groups of complex variety $M$, see e.g. \cite{GH}.

First we can apply the Lefschetz hyperplane theorem to the embedding $i: \Theta^n \subset A^{n+1}$, which claims that the homomorphism
$$
i^*: H^{p,q}(A^{n+1}) \to H^{p,q}(\Theta^n)
$$
is an isomorphism for $p+q\leq n-1$ and injective for $p+q=n$ (see \cite{GH}).

Since the Hodge numbers of abelian variety $A^{n+1}$ are 
$$
h^{p,q}(A^{n+1})={n+1 \choose  p} {n+1 \choose q}, \quad 0\leq p,q \leq n+1,
$$
we have
\beq{ht}
h^{p,q}(\Theta^{n})=h^{p,q}(A^{n+1})={n+1 \choose  p} {n+1 \choose q}, \quad p+q \leq n-1.
\eeq
By Serre duality $h^{p,q}(\Theta^{n})=h^{n-p,n-q}(\Theta^n)$, so this implies that
\beq{ht2}
h^{p,q}(\Theta^{n})={n+1 \choose  n-p} {n+1 \choose n-q}={n+1 \choose p+1} {n+1 \choose q+1}, \quad p+q \geq n+1.
\eeq
To compute the remaining Hodge numbers $h^{p,q}(\Theta^{n})$ with $p+q=n$ we can use now our formula (\ref{chip}):
$$
\chi^{p}(\Theta^{n})=\sum_{q=0}^n(-1)^q h^{p,q}(\Theta^n)=
(-1)^{n+p}A_{n+1,p}.
$$
In this sum the only unknown term is $h^{p,n-p}(\Theta^n)$. The straightforward calculations using the properties of  binomial coefficients show that the sum $S_{n,p}$ of the known terms is
\beq{sum}
S_{n,p}=(-1)^{p} {n+2 \choose  p+1} \left[(-1)^p\frac{2p-n}{n+2}{n+1 \choose  p}+\sum_{k=0}^{p-1}(-1)^k{n+1 \choose k}\right].
\eeq

As a result, we have the proof of Theorem 1.3 and the following formula for the Hodge numbers of the theta divisors.

\begin{prop} 
The Hodge numbers $h^{p,q}(\Theta^{n})$ of the theta divisor $\Theta^n$ with $p+q\neq n$ are given by (\ref{ht}), (\ref{ht2}), while when $p+q=n$ we have
\beq{hpn}
h^{p,n-p}(\Theta^{n})=A_{n+1,p}-S_{n,p},
\eeq
where $A_{n,p}$ are the Eulerian numbers and $S_{n,p}$ is given by (\ref{sum}).
\end{prop}

In particular, using formula (\ref{expli}) for the Eulerian numbers  we have
$$
A_{n,1}=2^n-(n+1), \quad A_{n,3}=3^n-2^n(n+1)+\frac{(n+1)(n+2)}{2},
$$
and thus
$$
h^{0,n}(\Theta^{n})=n+1, \quad h^{1,n-1}(\Theta^{n})=2^{n+1}-(n+2)+\frac{n^2(n+1)}{2},
$$
$$
h^{2,n-2}(\Theta^{n})=3^{n+1}-2^{n+1}(n+2)+ \frac{(n+1)(n+2)}{2}+\frac{n^3(n^2-1)}{12}.
$$

The Hodge diamonds of the theta divisors $\Theta^n$ for $n=2,3,4$ have the following form (with Betti numbers shown in the right column):

\medskip

 \[ \begin{tikzcd}[row sep=small, column sep=tiny]
    &   &&&&1&&&& &&  1\\ 
    &   &&&3&&3&&&&& 6 \\
    &   &&3&&10&&3&&&& 16 \\ 
    &   &&&3&&3&& &&& 6\\ 
    &   &&&&1&&&&&& 1 \\ 
    \\ 
\end{tikzcd} \]

 \[ \begin{tikzcd}[row sep=small, column sep=tiny]
    &   &&&&1&&&& &&  1\\ 
    &   &&&4&&4&&&&& 8 \\
    &   &&6&&16&&6&&&& 28 \\ 
    &   &4&&29&&29&&4 &&& 66\\ 
    &   &&6&&16&&6&&&& 28 \\ 
     &   &&&4&&4&&&&& 8 \\
     &   &&&&1&&&& &&  1\\ 
    \\ 
\end{tikzcd} \]

 \[ \begin{tikzcd}[row sep=small, column sep=tiny]
    &   &&&&1&&&& &&  1\\ 
    &   &&&5&&5&&&&& 10 \\
    &   &&10&&25&&10&&&& 45 \\ 
    &   &10&&50&&50&&10 &&& 120\\ 
    &   5&&66&&146&&66 &&5&& 288\\ 
    &   &10&&50&&50&&10 &&& 120\\ 
    &   &&10&&25&&10&&&& 45 \\ 
     &   &&&5&&5&&&&& 10 \\
     &   &&&&1&&&& &&  1\\ 
    \\ 
\end{tikzcd} \]

\section{Relation with permutohedral variety}

There is another natural algebraic variety related to the permutohedron, namely the corresponding toric variety $X_\Pi^n$ called {\it permutohedral}. 
Its normal fan corresponds to the standard $A_n$ hyperplane arrangement in $\mathbb R^{n+1}$ given by $x_i=x_j, \, 1\leq i < j\leq n+1$ with $x_1+\dots+x_{n+1}=0.$
In particular, $X_\Pi^1=\mathbb CP^1$, $X_\Pi^2$ is the degree 6 del Pezzo surface.  

The permutohedral varieties appeared in many relations. In particular, $X_\Pi^n$ is isomorphic to the Losev-Manin \cite{LM} compactification $\bar L_{0,n+3,2}$ of the moduli space $M_{0,n+3}$ (see more on this in \cite{BT-2024}).

Recall that toric variety can be constructed from any simple integer polytope $P^n$ (see \cite{Ful}). The topology of the permutohedral variety is being discussed in the literature (see e.g. the recent papers \cite{CL,LMP} and references therein). 
In particular, it is known that the Hodge numbers $h^{p,q}(X_\Pi^n)=0$ if $p\neq q$ and $h^{p,p}(X_\Pi^n)=h_p(\Pi_n)=A(n+1,p)$ are the Eulerian numbers, which is very different from what we have just seen for the theta divisors.

We claim that actually there is an interesting duality-like relation between the theta divisor $\Theta^n$ and permutohedral variety $X_\Pi^n$. Some evidence of such duality is given by the fact that the Todd genus $Td(X_\Pi^n) =1=(-1)^n Td(\Theta^n)$ and the Euler characteristic is the number of vertices of $\Pi^n$: $$\chi(X_\Pi^n) = (n+1)!=(-1)^n\chi(\Theta^n).$$ 

 We extends this to the proof of Theorem 1.4 claiming that the two-parameter Todd genus $Td_{s,t}(\Theta^n)$ of the theta divisor $\Theta^n$ and of the permutohedral variety $X_\Pi^n$  are different only by a sign:
\beq{duality}
Td_{s,t}(X_\Pi^n)=(-1)^n Td_{s,t}(\Theta^n).
\eeq



To prove this we use the results of T. Panov \cite{Panov}, who computed the $\chi_y$-genus of toric variety $X_P^n$ related to any simple polytope $P^n$ as the sum over vertices $p\in P^n$
$$
\chi_y(X_P^n)=\sum_{p}(-y)^{ind(p)},
$$
where $ind(p)$ is the index of $p$ with respect to generic height function on $P^n$ (see Theorem 3.1 in \cite{Panov}). Since it is known that the number of the vertices of index $k$ equals the coefficient $h_k(P^n)$ (see Khovanskii \cite{Kh1986}) we have that
$$
\chi_y(X_P^n)=\sum_{k=0}^nh_k(P^n)(-y)^k.
$$
This implies that
$$
Td_{s,t}(X_P^n)=h_{P^n}(-s,-t)=(-1)^n h_{P^n}(s,t),
$$
where $h_{P^n}(s,t)$ is the $h$-polynomial of the polytope $P^n$. 
Applying this to $P^n=\Pi^n$ and using our  Theorem 1.2 we have the relation (\ref{duality}).

The first part of Theorem 1.4 claims that the Betti number $b_{2k}(X_\Pi^n)$ coincides up to a sign with the Hirzebruch $\chi^k$-genus of the theta divisor $\Theta^n$:
\beq{relat1}
b_{2k}(X_\Pi^n)=(-1)^{n-k}\chi^k(\Theta^n),
\eeq
so that the Poincare polynomial $P(X_\Pi^n, s)=\sum_{i=0}^{2n}b_i(X_\Pi^n)s^i$ coincides up to a sign with $\chi_y$-genus of $\Theta^n$ with $y=-s^2$:
\beq{relat3}
P(X_\Pi^n, s)=(-1)^n\chi_{-s^2}(\Theta^n).
\eeq

Recall that by the general theory of toric varieties \cite{Ful} its even Betti number 
$b_{2k}(X_P^n)$ equals the coefficient $h_{k}(P^n)$ of the $h$-polynomial of the corresponding polytope $P$ (odd Betti numbers are zero). In our case of permutohedron $P=\Pi^n$ we have that
\beq{Eul}
b_{2k}(X_\Pi^n)=h_{k}(\Pi^n)=A(n+1,k)
\eeq
are the Eulerian numbers (cf. the formulae (\ref{betti1}), (\ref{betti2}) for the theta divisors).
Comparing this with Proposition 5.2 we have the relation (\ref{relat1}) and thus (\ref{relat3}).
This proves Theorem 1.4.


 In particular, for even $n$ using (\ref{tau}) we have the explicit formula for the signature $\tau(X_\Pi^n)$ in terms of Bernoulli numbers:
\beq{tauX}
\tau(X_\Pi^n)=\tau(\Theta^n)= \frac{2^{n+2}(2^{n+2}-1)}{n+2} B_{n+2}.
\eeq

 This suggests that the cobordisms classes of the permutohedral variety $X_\Pi^n$ and theta divisor $\Theta^n$ might be related by
$
[X_\Pi^n]=(-1)^n[\Theta^n].
$
However, this turns out to be true only for $n=1$ and $n=2.$ 
To see this we can use the results from the paper \cite{BPR10} by Buchstaber, Panov and Ray expressing the cobordism class of any toric variety in combination with our formula (\ref{CDnew}) for the Chern-Dold character \cite{BV2020}. In the case of the permutohedral variety we have the following formula.

\begin{Theorem} 
The cobordism class $X_\Pi^n$ of the permutohedral variety  can be expressed in terms of the cobordism classes of the theta divisors as
\beq{xpi}
[X_\Pi^n]=\sum_{\sigma \in S_{n+1}}\prod_{i=1}^n \frac{1}{\beta(t(z_{\sigma(i)}-z_{\sigma(i+1)}))}|_{t=0},
\eeq
where 
$
\beta(z)=z+\sum_{n=1}^\infty[\Theta^n]\frac{z^{n+1}}{(n+1)!}.
$
\end{Theorem}

In particular, this gives that $[X_\Pi^1]=-[\Theta^1],\, [X_\Pi^2]=[\Theta^2],$ but for $n=3$ the computer calculations using Wolfram Mathematica\footnote{We are grateful to Misha Kornev for helping us with this.} show that 
\beq{xpi3}
[X_\Pi^3]=\frac{1}{2}[\Theta^1]^3-\frac{2}{3}[\Theta^1][\Theta^2]-\frac{5}{6}[\Theta^3].
\eeq
Thus the link between these two classes of varieties does not go beyond the coincidence of generalised Todd genera, which looks even more mysterious.

There is another interesting parallel between the theta divisor $\Theta^n \subset A^{n+1}$ in abelian variety $A^{n+1}$ and open hypersurface $Z^n(\Pi^{n+1}) \subset T^{n+1}$ in the complex torus $T^n=(\mathbb C\setminus 0)^{n+1}$
given as the zero set $f(z)=0$ of a generic Laurent polynomial $f$ with permutohedral Newton polytope. The corresponding Hodge-Deligne numbers were computed by Danilov and Khovanskii in \cite{DKh}. It would be interesting to analyse their results in our context.


\section{Toda lattice and Tomei manifolds}

The (open) finite Toda lattice \cite{Flaschka, Moser} is the Hamiltonian system describing the interaction $n+1$ particles on the line with the Hamiltonian
$$H=\frac{1}{2}\sum_{i=1}^{n+1} p_i^2 +\sum_{j=1}^{n} e^{q_j-q_{j+1}},$$
In the Flaschka variables
$$
a_j=-\frac{1}{2}p_j, \quad j=1,\dots, n+1, \quad b_k=\frac{1}{2}e^{\frac{1}{2}(q_k-q_{k+1})}, \quad k=1,\dots,n
$$
the equations of motion take the algebraic form 
\beq{TF}
\dot{a}_j=2(b_j^2 -b_{j-1}^2), \quad \dot{b}_k =b_k(a_{k+1}-a_k)
\eeq
(we assume here that $b_0=b_{n+1}=0$).

A crucial observation due to Flaschka and Manakov is that the system (\ref{TF}) has the following Lax representation
\beq{Lax}
\dot L = [B, L],
\eeq
where
$$
L= \begin{pmatrix}
a_1  &  b_1 &    &    &    &  \\
b_1  &  a_2 &  b_2  & & & \\
   & \ddots &\ddots &  \ddots & \\
   & & b_{n-1} & a_{n}& b_{n}\\
   & & & b_{n} & a_{n+1}
\end{pmatrix},
B= \begin{pmatrix}
0  &  b_1 &    &    &    &  \\
-b_1  &  0 &  b_2  & & & \\
   & \ddots &\ddots &  \ddots & \\
   & & -b_{n-1} & 0 & b_{n}\\
   & & & -b_{n} & 0
\end{pmatrix}.
$$
This means that the eigenvalues of the matrix $L$ are preserved by the Toda flow. It is known that the coefficients of the characteristic polynomial $P_L(\lambda)=\det(L-\lambda  I)$ Poisson commute, proving that the Toda lattice is integrable in Liouville sense. The corresponding  set $M^n_+$ of the matrices $L$ with $b_i>0$ (called Jacobi matrices) with given spectrum $\Lambda=\{\lambda_1, \dots, \lambda_{n+1}\}$
is open and diffeomorphic to $\mathbb R^{n}$, so we do not have usual Liouville tori with quasiperiodic motion but instead the scattering (see the details in \cite{Moser}).

Following Tomei \cite{Tomei} consider the corresponding compact isospectral set
\beq{Tom}
M^{n}_T=\{L: spec \, L=\{\lambda_1, \dots, \lambda_{n+1}\}\}
\eeq
of all symmetric tridiagonal matrices $L$ (without restrictions that $b_i$ are positive), which we will call {\it Tomei manifold}.  For generic $\Lambda$ this is a smooth real manifold of dimension $n$, which is invariant under the (extended) Toda flow (\ref{TF}).  Tomei used this flow to study the topology of this manifold, which turned out to be quite interesting.\footnote{
Later Gaifullin \cite{Gaif} proved a remarkable fact that Tomei manifold can be used as a universal one in Steenrod's cycle realisation problem.}
In particular, he had shown that it admits the cell decomposition into $2^n$ permutohedra, corresponding to different choices of the signs of $b_i.$  For $n=2$ we have a surface of genus 2 glued from 4 hexagons (see \cite{Tomei}).

\begin{Theorem} (Tomei \cite{Tomei})
$M_T^{n}$ is an aspherical manifold with Euler characteristic
\beq{chit}
\chi(M_T^{n})=B_{n+2} \frac{2^{n+2}(2^{n+2}-1)}{n+2},
\eeq
where $B_n$ is $n$-th Bernoulli number.
\end{Theorem}

Comparing (\ref{chit})  with the formula  (\ref{tau}) for the signature $\tau(\Theta^n)$ of the theta divisor, we see that they coincide.

We extend this observation to the following result, demonstrating interesting relation of the Tomei manifold with $\Theta^n$ and $X_\Pi^n$.
Note that $M_T^n$ is real manifold of dimension $n$, while $\Theta^n$ and $X_\Pi^n$ are complex manifolds of real dimension $2n.$

Let $b_m(X)=\dim H^m(X, \mathbb Z_2)$ be the corresponding Betti numbers of a manifold $X.$ When the cohomology group $H^m(X, \mathbb Z)$ is torsion-free (which is the case for all three our manifolds), $b_m(X)$ is its rank.

\begin{Theorem} 
The numerical characteristics of the Tomei manifold $M_T^n$, theta divisor $\Theta^n$ and permutohedral variety $X_\Pi^n$ are related by
\beq{relat1X}
b_k(M_T^n)=b_{2k}(X_\Pi^n)=(-1)^{n-k}\chi^k(\Theta^n).
\eeq
In particular, the Euler characteristic of $M_T^n$ equals the signatures of $X_\Pi^n$ and $\Theta^n$:
\beq{chitX}
\chi(M_T^{n})=\tau(X_\Pi^n)=\tau(\Theta^n).
\eeq
\end{Theorem}

\begin{proof}
The Betti numbers of Tomei manifold were computed by Fried \cite{Fried}, who showed that
$
b_k(M_T^n)=A(n+1,k),
$
where $A(n,k)$ are Eulerian numbers. Comparing this with (\ref{Eul}) and (\ref{relat1}), we have (\ref{relat1X}).

A more conceptual proof of this follows from the theory of {\it small covers of simple polytopes} from Davis and Januszkiewicz \cite{DJ}.
The Tomei manifold $M_T^n$ corresponds to the case when the polytope is permutohedron $\Pi^n$ for certain characteristic function, which can be interpreted as colouring the faces of permutohedron in $n$ colours (see \cite{DJ,Gaif}).
Theorem 3.1 from \cite{DJ} says that the Betti number $b_k(M_P)$ (over $\mathbb Z_2$) of a small cover of simple polytope $P$ equals the coefficient $h_k(P)$ of the corresponding $h$-polynomial.
In our case this implies that 
$b_k(M_T^n)=h_k(\Pi^n),$
and thus (\ref{relat1X}).
 
 To prove that $\chi(M_T^{n})=\tau(X_\Pi^n)$ we use the general result from the theory of toric varieties \cite{Panov} (see also \cite{LR}) that the signature of such variety $X_\Pi^n$ is the alternating sum of the even Betti numbers:
\beq{LR}
 \tau(X_\Pi^n)=\sum_{k=0}^n(-1)^k b_{2k}(X_\Pi^n).
 \eeq
The equality $\tau(X_\Pi^n)=\tau(\Theta^n)$ follows now from (\ref{tauX}).
\end{proof}

Let us consider now the {\it Hermitian Tomei manifold} $M^{2n}_{HT}$ as the set of Hermitian tridiagonal matrices
$$
L^H= \begin{pmatrix}
a_1  &  b_1 &    &    &    &  \\
\bar b_1  &  a_2 &  b_2  & & & \\
   & \ddots &\ddots &  \ddots & \\
   & & \bar b_{n-1} & a_{n}& b_{n}\\
   & & & \bar b_{n} & a_{n+1}
\end{pmatrix},
$$
with given spectrum $Spec \, L=\Lambda=(\lambda_1,\dots, \lambda_{n+1})$ (known to be real), where $a_k\in \mathbb R$ and $b_j\in \mathbb C$. 
For generic $\Lambda$ this is a smooth submanifold of the set $O_\Lambda$ of all Hermitian matrices 
with spectrum $\Lambda$,  which can be viewed as a coadjoint orbit $U(n+1)/T^{n+1}$ of the unitary group $U(n+1).$

Note that the embedding $M_T^n \subset M^{2n}_{HT}$ is equivariant with respect to the natural actions of $\mathbb Z_2^{n}$ and  $T^{n},$ where  $T^{n}$ is the group of diagonal matrices from $SU(n+1)$ and $\mathbb Z_2^{n}\subset T^n$ is its subgroup with $\pm 1$ on the diagonal.


Bloch, Brockett and Ratiu \cite{BBR} had shown that the Toda flow is gradient for some metric on $M_T^n$ and the height function $tr (\rho L), \, \rho=diag (1, \dots, n+1)$, so that Tomei results \cite{Tomei} can be interpreted within the classical Morse theory \cite{Milnor}. Using this one can obtain covering of $M_T^n$ by $(n+1)!$ open charts and check that they satisfy the properties of the small cover in terminology  of Davis and Januszkiewicz \cite{DJ}. 

In the Hermitian case one can use the results of Bloch, Flaschka and Ratiu \cite{BFR} to deduce that $M^{2n}_{HT}$ is a toric manifold (in the sense of Davis and Januszkiewicz) with the same orbit space $\Pi^n$ and the same characteristic function as in the real Tomei case (see \cite{DJ,Gaif}). It is natural to compare it with the permutohedral variety $X^{n}$.

 

 
\begin{Theorem} 
Hermitian Tomei manifold $M^{4n}_{HT}$ is not homotopically equivalent (and hence not diffeomorphic) to the permutohedral variety $X^{2n}$.

In addition, $M^{4n}_{HT}$ is not equivariantly diffeomorphic to any symplectic manifold $M^{4n}$ with Hamiltonian action of torus $T^{2n}.$
\end{Theorem}

\begin{proof}
Davis and Januszkiewicz \cite{DJ} proved that $M^{2m}_{HT}$ is stably parallelisable, so due to Hirzebruch \cite{Hirz} the signature $\tau(M^{4n}_{HT})=0.$ On the other hand, from (\ref{tauX}) we see that $\tau(X^{2n})\neq 0.$  Since the signature is homotopic invariant, we conclude that $M^{2n}_{HT}$ and $X^{2n}$ are not homotopically equivalent.

To prove the second part, we use the results of Delzant \cite{Del}, which imply that that every manifold $M^{4n}$ with Hamiltonian action of torus $T^{2n}$ is equivariantly diffeomorphic (but, in general, not symplectomorphic) to an algebraic complex toric variety $Y^{2n}$ with combinatorially equivalent moment polytope. Panov \cite{Panov} (see also \cite{LR}) proved that the signature $\tau(Y^{2n})$ depends only on combinatorics of the corresponding polytope (which in our case is permutohedron), so $\tau(Y^{2n})=\tau(X^{2n})\neq 0.$ Since the signature of $M^{4n}_{HT}$ is zero, it cannot be diffeomorphic to $M^{4n.}$ 
\end{proof}

Note that $M_{HT}^2=S^2$ is two-dimensional sphere with the standard symplectic structure and a natural Hamiltonian action of $T^1=S^1$, so our result cannot be extended to all dimensions.

Our theorem explains why Bloch, Flaschka and Ratiu \cite{BFR} considered the embedding into the coadjoint orbit $O_\Lambda$ only of the ``isospectral set" $\mathfrak J_\Lambda$, but not of the ``full isospectral manifold" $M^{2n}_{HT}$ (see the comments at the end of Section 2.2 in \cite{BFR}).

For $n=2$ we can claim a stronger result (cf. Section 6 in Hirzebruch \cite{Hirz0} and Chapter 9 in Buchstaber, Panov \cite{BP}).

\begin{Theorem} 
Hermitian Tomei manifold $M^{4}_{HT}$ does not admit any almost complex (and hence, any symplectic) structure. 

In particular, there is no embedding of $M^{4}_{HT}$ into the coadjoint orbit $O_\Lambda$ with non-degenerate restriction of the canonical symplectic form on $O_\Lambda.$
\end{Theorem}

\begin{proof}
Assume that $M^4=M^{4}_{HT}$ has an almost complex structure, then we have the canonically defined orientation and thus the fundamental cycle $<M^4> \in H_4(M^4, \mathbb Z).$ 
For any almost complex manifold we have well defined Chern numbers of such manifold as the values of the corresponding Chern classes on the fundamental cycle $<M^4>.$ 
In terms of these numbers one can express the Euler characteristic, signature and Todd genus of any almost complex manifold $M^4$ as follows \cite{Hirz}
$$
\chi(M^4)=c_2, \quad \tau(M^4)=\frac{c_1^2-2c_2}{3}, \quad Td(M^4)=\frac{c_1^2+c_2}{12}.
$$
As a result for any almost complex manifold $M^4$ we have the relation
$
Td(M^4)=\frac{1}{4}(\tau(M^4)+\chi(M^4)).
$
From the results of \cite{DJ} the Euler characteristic $\chi(M^4)=(2+1)!=6$ and since the signature $\tau(M^4)=0$ we have
$
Td(M^4)=\frac{6+0}{4}=\frac{3}{2}.
$
This contradicts the classical Hirzebruch result \cite{Hirz0} that any almost complex manifold must have integer Todd genus.
Since any symplectic manifold admits an almost complex structure, we conclude that $M^{4}_{HT}$ has no symplectic structures.
\end{proof}

Finally, let us discuss the Hermitian Tomei manifold in the context of complex cobordisms.
Recall that {\it $U$-structure} on a real manifold $M^m$ is an isomorphism of real vector bundles
\beq{U}
TM^m\oplus (2N-m)_{\mathbb R}\cong r\xi,
\eeq
where $TM^m$ is the tangent bundle of $M^m$, $(2N-m)_{\mathbb R}$ is trivial real $(2N-m)$-dimensional bundle over $M^m$, $\xi$ is a complex vector bundle over $M^m$ and $r\xi$ is its real form.
Buchstaber and Ray \cite{BR01} showed that any smooth toric manifold (in particular, $M^{2n}_{HT}$) can be supplied with a canonical $U$-structure, which is invariant under the $T^n$-action ({\it BR-structure}).

\begin{Theorem} 
As a $U$-manifold with BR-structure $M^{2n}_{HT}$ has the zero complex cobordism class
and does not admit any $T^n$-invariant almost complex structure. 
\end{Theorem}

\begin{proof}
We use the results of Buchstaber, Panov and Ray \cite{BPR10}, who provided a formula for the cobordism class of any smooth toric $U$-manifold with the BR-structure (see Theorem 5.16 and Corollary 4.9 in \cite{BPR10}).
To apply formula (4.10) from that paper, we need to find the signs of the vertices of permutohedron, corresponding to BR-structure. Since the characteristic function in our case comes from colouring of the faces, it is easy to see that the neighbouring vertices of permutohedron have opposite signs. This means that the total sum in the right hand side of formula (4.10) (and hence the cobordism class of $M^{2n}_{HT}$) is zero: $[M^{2n}_{HT}]=0.$

If $M^{2n}_{HT}$ would admit $T^n$-invariant almost complex structure then in formula (4.10) all signs would be plus, which leads to a contradiction.
\end{proof}

When $n=1$ the manifold $M_{HT}^2$ can be identified with $\mathbb CP^1$, but with different $U$-structure. As a complex manifold $\mathbb CP^1$ is $U$-manifold with $N=2$ and $\xi=\bar\eta\oplus\bar\eta$, where $\eta$ is the tautological line bundle over $\mathbb CP^1$ and $\bar \eta$ is its dual, while the BR-structure on $M_{HT}^2$ corresponds to $N=2$ and different choice of $\xi=\eta\oplus\bar\eta$ in (\ref{U}). The BR-structure on $M_{HT}^2$ comes naturally from the representation of $S^2$ as the quotient of the unit quaternion sphere $S^3=\{q \in \mathbb H, \, |q|=1\}$ by the action of $S^1=\{z \in \mathbb C, \, |z|=1\} \subset \mathbb H$ given by the left multiplication  $q \to zq.$ If we identify $\mathbb H$ with $\mathbb C^2$ using $q=z_1+j z_2$ then $S^1$ acts with the matrix $diag\,(z,\bar z)$ (in contrast with the multiplication by $z$ in the $\mathbb CP^1$ case).

\section{Acknowledgements.}
We are grateful to Alexander Gaifullin, Alexander Kuznetsov, Taras Panov, Yuri Prokhorov and Artie Prendergast-Smith for very helpful discussions.

\end{document}